\documentclass{article}
\usepackage[T1]{fontenc}

\usepackage{chao}
\usepackage{graphicx}
\usepackage{cleveref}

\renewcommand{\phi}{\varphi}

\DeclareMathOperator{\stab}{stab}

\crefname{problem}{Problem}{Problems}
\crefname{conjecture}{Conjecture}{Conjectures}
\usepackage{thmtools}
\usepackage{thm-restate}
\begin{document}
\date{}
\title{On the Congruency-Constrained Matroid Base}

\author{Siyue Liu\thanks{\email{siyueliu@andrew.cmu.edu}, Carnegie Mellon University.} \and Chao Xu\thanks{\email{the.chao.xu@gmail.com}, University of Electronic Science and Technology of China.}}

\maketitle              
\begin{abstract}
Consider a matroid where all elements are labeled with an element in $\mathbb{Z}$. We are interested in finding a base where the sum of the labels is congruent to $g \pmod m$. We show that this problem can be solved in $\tilde{O}(2^{4m} n r^{5/6})$ time for a matroid with $n$ elements and rank $r$, when $m$ is either the product of two primes or a prime power. The algorithm can be generalized to all moduli and, in fact, to all abelian groups if a classic additive combinatorics conjecture by Schrijver and Seymour holds true. We also discuss the optimization version of the problem.
\end{abstract}
%
%
%



\section{Introduction}

Recently, there has been a surge of work on congruency-constrained combinatorial optimization problems, such as submodular minimization \cite{nagele_submodular_2019}, constraint satisfaction problems \cite{BrakensiekGG22}, and integer programming over totally unimodular matrices \cite{ArtmannWZ17,NageleSZ22}. In this paper, we consider congruency-constrained matroid base problems.

As a motivation, consider the exact matching problem, which asks for a perfect matching in a red-blue edge-colored bipartite graph with exactly $k$ red edges \cite{PapadimitriouY82}. A more general variant, where each edge has a weight of polynomial size and the goal is to find a perfect matching with an exact weight, has a randomized polynomial-time algorithm \cite{10.1145/28395.383347}. Exact matching is a special case of the exact matroid intersection problem, where the goal is to find a common base of a particular weight. Webb provided an algebraic formulation of the exact base problem that leads to an efficient solution when the matroid pair is Pfaffian and the weights are small \cite{WebbKerri2004}. For a single matroid, the exact base problem is defined as follows.

\begin{problem}[Exact Matroid Base]
Given a matroid $M=(E,\mathcal{I})$, a natural number $t$, and a label function $\ell: E \to \mathbb{N}$. Find a base $B$ in $M$ such that $\sum_{x \in B} \ell(x) = t$, if one exists.
\end{problem}

The problem is NP-hard, as the classical subset sum problem can be reduced to it. However, the subset sum and its optimization variant, the knapsack problem, can be solved in pseudopolynomial time \cite{Bellman1956}. Thus, Papadimitriou and Yannakakis questioned whether the exact matroid base could be solved in pseudopolynomial time \cite{PapadimitriouY82}. Barahona and Pulleyblank provided an affirmative answer for graphic matroids \cite{barahonaExactArborescencesMatchings1987}. Subsequently, Camerini and Maffioli demonstrated that for matroids representable over the reals, there exists a randomized pseudopolynomial time algorithm for the exact matroid base problem \cite{CameriniGM13}. Recently, Doron-Arad, Kulik, and Shachnai showed that pseudopolynomial time algorithms for arbitrary matroids cannot exist, demonstrated through carefully constructed paving matroids \cite{doronarad2023tight}. For one optimization variant, where the goal is to maximize the value of a base while adhering to an upper bound on the budget, some approximation results are known \cite{matroidalknapsack,multi-constrained_1989}.

Recent progress has involved relaxing the exact constraint to a modular constraint to achieve interesting advancements. For instance, finding a perfect matching with an even number of red edges has been shown to be solvable in deterministic polynomial time \cite{elmaalouly_et_al:LIPIcs.ISAAC.2023.28}. Hence, we relax the exact constraint and study the following problem:

\begin{problem}[$m$-Congruency-Constrained Matroid Base ($\operatorname{CCMB}(m)$)]
Let $M=(E,\mathcal{I})$ be a matroid, $g \in \{0, \ldots, m-1\}$, and $\ell:E \to \mathbb{Z}$ be a label function. Find a base $B$ in $M$ such that $\sum_{x\in B} \ell(x) \equiv g \pmod{m}$, if any.
\end{problem}

As far as we know, except for results implied by the exact matroid base problem, little is known about this problem. An algorithm with a running time polynomial in $m$ and the size of the matroid cannot exist, as it would imply the existence of a pseudopolynomial time algorithm for the exact matroid base problem. We aim for the next best thing: does there exist a fixed-parameter tractable (FPT) algorithm for $\operatorname{CCMB}(m)$, parameterized by $m$?

Similar to \cite{naegele_2023_advances}, we consider finite abelian groups instead of solely congruency constraints. For a finite abelian group $G$, let $\ell:E\rightarrow G$ be a label function mapping each element of the matroid to an element of $G$. $\ell$ is called a \emph{$G$-labeling} of $E$. 
A base with label sum $\sum_{x\in B} \ell(x)=g$ is called a \emph{$g$-base}. 
Our goal can be summarized as determining the complexity of the following two problems regarding finding $g$-bases, parameterized by some finite abelian group $G$.

\begin{problem}[Group-Constrained Matroid Base ($\operatorname{GCMB}(m)$)]\label{prob:feas}
Given $g \in G$, a matroid $M = (E,\mathcal{I})$, and a label function $\ell:E \rightarrow G$, find a $g$-base, if one exists.
\end{problem}

\begin{problem}[Group-Constrained Optimum Matroid Base ($\operatorname{GCOMB}(m)$)]\label{prob:opt}
Given $g \in G$, a matroid $M = (E,\mathcal{I})$, a label function $\ell:E \rightarrow G$, and a weight function $w:E \rightarrow \mathbb{R}$, find a $g$-base of minimum weight, if one exists.
\end{problem}

The closest study related to matroid bases and group constraints comes from the additive combinatorics community, where theorems considering the number of different labels attainable by the bases have been discovered \cite{schrijver1990spanning,devos2009generalization}. Recently, H\"orsch et al. \cite{horsch2024problems} study generalized problems where the label sum of the bases is allowed to take a subset of $G$ instead of a single element. The same problem for the common bases of two matroids is also studied.

\paragraph{Our Contribution.}
We show that $\operatorname{GCMB}(m)$ can be solved in $\tilde{O}(2^{4|G|}n r^{5/6})$ time for a group $G$, if either the size of $G$ equals the product of two primes, or $G$ is a cyclic group of size a power of a prime. In particular, this proves that $\operatorname{CCMB}(m)$ is in FPT parameterized by $m$ if $m$ is the product of two primes or a prime power. 

These results are obtained through proximity results. For this, let $\mathcal{B}$ be the set of bases of $M$. Let $\mathcal{B}(g)$ be the collection of all $g$-bases of $M$. For a weight function $w:E \rightarrow \mathbb{R}$, denote by $\mathcal{B}^* := \arg\min_{B \in \mathcal{B}} w(B)$ and $\mathcal{B}^*(g) := \arg\min_{B \in \mathcal{B}(g)} w(B)$ the sets of \emph{optimum bases} and \emph{optimum $g$-bases}, respectively. We introduce the following key concepts of this paper.
\begin{definition}
    A $G$-labeling $\ell$ of matroid $M$ is \textbf{$k$-close} if for every $g \in G$ such that a $g$-base exists, every base has a $g$-base differing from it by at most $k$ elements. If every $G$-labeling is $k$-close for every matroid, then $G$ is \textbf{$k$-close}.
\end{definition}

\begin{definition}
    A $G$-labeling $\ell$ of matroid $M$ is \textbf{strongly $k$-close} if for every weight $w: E\rightarrow \R$ and $g \in G$ such that a $g$-base exists, every optimum base has an optimum $g$-base differing from it by at most $k$ elements. If every $G$-labeling is strongly $k$-close for every matroid, then $G$ is \textbf{strongly $k$-close}.
\end{definition}

Note that being strongly $k$-close implies being $k$-close by setting $w \equiv 0$. We prove that $G$ is $(|G|-1)$-close if either the size of $G$ equals the product of two primes, or $G$ is a cyclic group of size a power of a prime (\Cref{cor:closeness}). More generally, we show that every group is $(|G|-1)$-close assuming a conjecture by Schrijver and Seymour (\Cref{conj:schrijver_seymour}). This gives rise to the desired FPT algorithm for $\operatorname{GCMB}(m)$ (\Cref{thm:FPT}). Lastly, we prove the strong $k$-closeness for strongly base orderable matroids (\Cref{thm:stronglyorderable}) and small groups, for some $k\leq |G|$ which only depends on $G$. These lead to FPT algorithms for $\operatorname{GCOMB}(m)$ for those matroids and groups.

\paragraph{Overview.}
In \Cref{sec:prelminaries}, we define the problem and provide the necessary background in matroid and group theory. In \Cref{sec:gcomb}, we discuss established algorithms and demonstrate the existence of an FPT algorithm for $\operatorname{GCMB}(m)$ and $\operatorname{GCOMB}(m)$, provided that $G$ is $k$-close and strongly $k$-close, respectively, for some $k$ which only depends on $G$. Starting
 from Section \ref{sec:minexample}, we prove the proximity results. In \Cref{sec:minexample}, we consider what would constitute the minimum counterexample to the statement that $G$ is (strongly) $k$-close. We observe that the minimum counterexample would have to be a block matroid of rank $k+1$. In \Cref{sec:kclose}, we prove the $(|G|-1)$-closeness for certain groups. Finally, \Cref{sec:stronglykclose} discusses the optimization version and results on strong $k$-closeness.

\paragraph{Notation.} We are given a matroid $M=(E,\mathcal{I})$, a group $G$, and a label function $\ell: E\rightarrow G$. For any $F\subseteq E$, denote by $\ell(F):=\sum_{e\in F} \ell(e)$. For a subset $H\subseteq G$, denote by $E(H):=\ell^{-1}(H)=\{e\in E\mid \ell(e)\in H\}$ all the elements with labels in $H$. If $H=\{g\}$ is a singleton, we shorten $E(\{g\})$ to $E(g)$. Denote by $\ell(M)$ the set $\{\ell(B)\mid B$ is a base of $M\}$. For $F\subseteq E$, denote the rank of $F$ in $M$ by $r_M(F)$, where $M$ in the rank function $r_M(\cdot)$ can be omitted if the context is clear. Denote the rank of matroid $M$ by $r(M):=r_M(E)$.

\section{Preliminaries}\label{sec:prelminaries}

Let $M=(E,\mathcal{I})$ be a matroid,  where $E$ is the ground set and $\mathcal{I}$ is the family of independent sets. For a subset $F \subseteq E$, denote by $M \setminus F$ the \emph{deletion} minor of $M$ defined on the ground set $E\setminus F$ whose independent sets are those of $M$ restricted to $E \setminus F$. If $F \subseteq E$ is an independent set of $M$, denote by $M/F$ the \emph{contraction} minor of $M$ defined on the ground set $E\setminus F$. A subset $A \subseteq E \setminus F$ is independent in $M/F$ if and only if $A \cup F$ is independent in $M$. Their rank functions satisfy the relation $r_{M/F}(A) = r_M(A \cup F) - r_M(F)$. We assume the matroid is loopless throughout, which means $\{e\} \in \mathcal{I}$ for any $e \in E$.

A matroid possesses the base exchange property: For any two bases $A$ and $B$, there exists a series of elements $a_1, \ldots, a_k \in A \setminus B$ and $b_1, \ldots, b_k \in B \setminus A$ such that $A - \{a_1,\ldots,a_j\} + \{b_1,\ldots,b_j\}$ is also a base for each $j = \{1, \ldots, k\}$, and $A - \{a_1,\ldots,a_k\} + \{b_1,\ldots,b_k\} = B$. The distance $d(A,B)$ between two bases $A$ and $B$ is the number of steps needed to exchange elements from $A$ to $B$, which equals $|A \setminus B|$.

A matroid is \emph{strongly base orderable} if for any two bases $A$ and $B$, there is a bijection $f:A\to B$ such that for any $X\subseteq A$, we have $A-X+f(X)$ is a base. A function with the previous property can be taken to satisfy $f(a)=a$ for all $a\in A\cap B$. 
A matroid is a \emph{block matroid} if the ground set is the union of two disjoint bases, and such disjoint bases are called \emph{blocks}. 

Given a group $G$ and a normal subgroup $H\subseteq G$, denote by $G/H$ the quotient group consisting of equivalence classes of the form $gH:=\{g+h\mid h\in H\}$, which are called the \emph{cosets} of $H$, where $g_1$ is equivalent to $g_2$ if and only if $g_1^{-1}g_2\in H$. For a subset $F\subseteq G$, the \emph{stabilizer} of $F$ is defined by $\stab(F):=\{g\in G\mid g+F=F\}$. It is easy to see that $\stab(F)$ is a subgroup of $G$.

The Davenport constant \cite{davenport1966midwestern,olson1969combinatorial} of $G$, denoted by $D(G)$, is the minimum value such that every sequence of elements from $G$ of length $D(G)$ contains a non-empty subsequence that sums to $0$. In other words, the longest sequence without non-empty subsequence that sums to $0$ has length $D(G)-1$. 
By the fundamental theorem of finite abelian groups, any finite abelian group $G$ can be decomposed into $G=\Z_{m_1}\times \ldots \times \Z_{m_r}$ where $1\mid m_1 \mid m_2\mid...\mid m_r$. Here $\Z_m$ is the group of integers modulo $m$, and $\times$ is the group direct product. It is shown that a lower bound of $D(G)$ is $M(G) := \sum_{i}^r (m_i - 1) + 1$. It was proved independently by Olson \cite{olson1969combinatorial} and Kruswijk \cite{baayen1968een} that $D(G) = M(G)$ for \emph{$p$-groups}, in which the order of every element is a power of $p$, and for $r=\{1,2\}$. 
It also holds that $D(G)\leq |G|$, where equality holds for all cyclic groups \cite{van1969combinatorial}.

Our paper makes progress on the following two conjectures. {\footnote{In an earlier preprint version of this paper, we conjectured that every finite abelian group is (strongly) $(D(G)-1)$-close, but this was recently disproved by \cite{horsch2024problems}.}}
\begin{conjecture}[Feasibility]\label{conj:feasibility}
Every finite abelian group $G$ is $(|G|-1)$-close. 
\end{conjecture}
\begin{conjecture}[Optimization]\label{conj:optimization}
Every finite abelian group $G$ is strongly $(|G|-1)$-close. 
\end{conjecture}

Observe the conjectures are best possible when the group is cyclic. Indeed, the tight example is a block matroid of size $|G|-1$, where the first block has all its element labeled $1$ and the second has all its elements labeled $0$. Then, the closest $0$-base to the first block is the second block, which has distance $|G|-1$ from it.

\section{Algorithms for GCOMB}\label{sec:gcomb}
In this secton, we show how $k$-closeness and strong $k$-closeness of a group $G$ lead to FPT algorithms for $\operatorname{GCMB}(m)$ and $\operatorname{GCOMB}(m)$, respectively. In particular, \Cref{thm:FPT} shows that if there is a function $f$ such that $G$ is $f(G)$-close or strongly $f(G)$-close, then for a fixed group $G$, $\operatorname{GCMB}(m)$ or $\operatorname{GCOMB}(m)$ is in FPT, respectively. 

We assume that the matroid can be accessed through an independence oracle that takes a constant time per query. Each group operation also takes $O(1)$ time.

There is a folklore algorithm which shows for any fixed group $G$, $\operatorname{GCOMB}(m)$ is solvable in polynomial time by reducing it to a polynomial number of matroid intersections. 

Let $M=(E,\mathcal{I})$ be a matroid of size $n$ and rank $r$. Recall $E(g)=\ell^{-1}(g)$ for $g\in G$. For a base $B$ of matroid $M$, the vector $a\in \Z_{\geq 0}^G$ such that $a_g = |B\cap E(g)|$ is called the \emph{signature} of $B$.
For a target label $h\in G$, every $h$-base $B$ satisfies $h=\sum_{g\in G} a_g\cdot g$ for the signature $a$ of $B$. Here $a \cdot g$ is defined to be $\sum_{i=1}^{a} g$. 
For a fixed vector $a$ with $\sum_{g} a_g = r$, $0\leq a_g\leq |E(g)|$ and $\sum_{g} a_g \cdot g = h$, we can check if any base has signature $a$, in fact, find a minimum cost base with signature $a$. Indeed, one can obtain this through matroid intersection with a partition matroid, whose partition classes are $\{E(g)\mid g\in G\}$. It suffices to find an optimum base of matroid $M$ such that its intersection with $E(g)$ has cardinality $a_g$ for each $g\in G$.
So the number of vectors we need to check is bounded above by ${r+|G|-1 \choose |G|-1} \leq r^{|G|-1}$. Hence, we have to run the matroid intersection algorithm $r^{|G|-1}$ times in order to find the optimum $g$-base. Currently, the fastest algorithm for matroid intersection has running time $O(nr^{5/6} \log r)$ time assuming independence oracle takes constant time \cite{blikstad:LIPIcs.ICALP.2021.31}. Hence the total running time is $\tilde{O}(nr^{O(|G|)})$ which is polynomial as $|G|$ is fixed.
Although the algorithm is polynomial time for a fixed $G$, if we parameterize by $G$, then this algorithm is not FPT.

If $G$ is (strongly) $k$-close for a constant $k$, then there is also a polynomial time algorithm: find an (optimum) base $B$, consider all possible bases $B'$ such that $|B\setminus B'|\leq k$, and return an (optimum) base $B'$ with the desired label. The running time is $O(r^kn^k)$ by enumerating all bases $k$-close to $B$.

We show that (strong) $k$-closeness gives rise to an FPT algorithm parametrized by $G$ by combining it with the matroid intersection approach. 

\begin{theorem}\label{thm:FPT}
If $G$ is $k$-close or strongly $k$-close, then there is a $\tilde{O}(\binom{k+|G|-1}{k}^2 nr^{5/6})$ running time algorithm for $\operatorname{GCMB}(m)$ or $\operatorname{GCOMB}(m)$, respectively.
\end{theorem}
\begin{proof}
The algorithm consists of two parts. First find any base $B$. By (strong) $k$-closeness, we know there is a feasible (optimum) $g$-base $D$ that differs from $B$ in at most $k$ positions. Let $a_D$ and $a_B$ be the signature vectors of $D$ and $B$, respectively. 
Let $a^+:=(a_D-a_B)\vee 0$, $a^-:=(a_D-a_B)\wedge 0$ be vectors taking coordinate-wise maximum and minimum of $(a_D-a_B)$ with $0$, respectively. Note that $a^+\geq 0$ and $\sum_{g}a^+_g\leq k$. Therefore, there are at most $\binom{k+|G|-1}{k}$ possible $a^+$'s, and similarly at most $\binom{k+|G|-1}{k}$ possible $a^-$'s.
Hence, we only need to run the matroid intersection algorithm for at most $\binom{k+|G|-1}{k}^2$ times to search for the (optimum) $g$-bases which are (strongly) $k$-close to $B$. 
\qed
\end{proof}


If we take $k=|G|-1$ and use the fact that $\binom{2k}{k}\leq 2^{2k}$, we obtain the following corollary.

\begin{corollary}\label{thm:FPT2}
If $G$ is $(|G|-1)$-close or strongly $(|G|-1)$-close, then there is an algorithm whose running time in $\tilde{O}(2^{4|G|}nr^{5/6})$ for $\operatorname{GCMB}(m)$ or $\operatorname{GCOMB}(m)$, respectively.
\end{corollary}

\section{Properties of minimum counterexamples to \Cref{conj:feasibility} and \Cref{conj:optimization}}\label{sec:minexample}

In this section, we show that to prove \Cref{conj:feasibility} or \Cref{conj:optimization}, it suffices to prove it for a restricted type of matroids. Assume for the sake of contradiction that $G$ is not (strongly) $k$-close, then there exists a counterexample of minimum size. This section shows that the minimum counterexamples are block matroids of size $(k+1)$. Hence, it suffices to prove the conjectures for those matroids. We need a lemma by Brualdi \cite{brualdi_1969}. 
\begin{lemma}[\cite{brualdi_1969}]\label{lem:braldi}
Let $A$ and $B$ be bases in a matroid, then there exists a bijection $f:A\setminus B\to B\setminus A$, such that $A-a+f(a)$ is a base.
\end{lemma}
We give the following lemma.
\begin{lemma}\label{lemma:optimum-dist1}
    Given a matroid $M=(E,\mathcal{I})$ and any weight function $w:E\rightarrow \R$, let $A$ be an optimum base such that its weight is minimized. Then, for any $a\in A$, there exists some $b\in E\setminus A$ such that $A-a+b$ is an optimum base in $M'=M\setminus a$ with weight being the restriction of $w$ to $E\setminus a$.
\end{lemma}
\begin{proof}
    Suppose $A'$ is an optimum base in $M\setminus a$. By \Cref{lem:braldi}, there exists a bijection $f:A\setminus A'\rightarrow A'\setminus A$ such that $A-e+f(e)$ is a base for any $e\in A\setminus A'$. Since $A$ is an optimum base, we have $w(A)\leq w(A-e+f(e))$, and thus $w(e)\leq w(f(e))$ for each $e\in A\setminus A'$. For any $a\in A$, let $b=f(a)$. Then, $w(A-a+b)=w(A)-w(a)+w(b)\leq w(A)-\sum_{e\in A\setminus A'}(w(e)-w(f(e)))=w(A')$. The inequality is because $a\in A\setminus A'$ and $w(e)\leq w(f(e))$ for any $e\in A\setminus A'\setminus a$. Since $A'$ is an optimum base in $M\setminus a$, $A-a+b$ is also an optimum base in $M\setminus a$.
\qed
\end{proof}

Next, we show that if $G$ is not strongly $k$-close, then there has to be a $G$-labeling $\ell$, a weight function $w$, an element $g\in G$, and a block matroid with blocks $A,B$, such that $A$ is an optimum base, $B$ is the closest optimum $g$-base to $A$, and $d(A,B)>k$. We call the pair of bases $A,B$ a \emph{witness}.

\begin{theorem}\label{thm:mim_counterexample_weight}
If a finite abelian group $G$ is not strongly $k$-close, then the there is a block matroid $M$ of rank $k+1$ with blocks $A$ and $B$ that form a witness.
\end{theorem}
\begin{proof}
Let matroid $M=(E,\mathcal{I})$ be the smallest matroid, in terms of the number of elements, such that $G$ is not strongly $k$-close for $M$. Let $w:E\rightarrow \Z$ be a weight function and $g\in G$. Recall that $\mathcal{B}^* = \arg\min_{B\in \mathcal{B}} w(B)$ and $\mathcal{B}^*(g) = \arg\min_{B\in \mathcal{B}(g)} w(B)$ are the sets of optimum bases and optimum $g$-bases, respectively. There exists $A\in\mathcal{B}^*$ and $B\in\mathcal{B}^*(g)$ such that $B$ is an optimum $g$-base closest to $A$ but $|A\setminus B|>k$.

First, we argue that $A\cap B=\emptyset$. If not, let $M'=(E',\mathcal{I}')$ be the matroid obtained by contracting $A\cap B$, and let $A'=A\setminus B$, $B'=B\setminus A$, $g'=g-\ell(A\cap B)$ and $w'=w|_{E'}$. Clearly, $A'$ is a weighted minimum base of $M'$ and $B'$ is a weighted minimum $g'$-base of $M'$ with weight $w'$. Also, $d(A',B')=|A'\setminus B'|=|A\setminus B|>k$, and $B'$ is an optimum $g'$-base closest to $A'$. However, $|E'|=|E|-|A\cap B|<|E|$, contradicting to the minimality of $|E|$. Moreover, $A\cup B=E$. Otherwise, let $M'$ be the matroid obtained by deleting $E\setminus (A\cup B)$ and $A,B$ stays a witness in $M'$, contradicting to the minimality of $|E|$. Therefore, $M$ is a block matroid which is a block matroid with blocks $A$ and $B$.

Suppose $r(M)>k+1$. Pick any $a\in A$. Let $M'=(E',\mathcal{I}'):=M\setminus a$. By \Cref{lemma:optimum-dist1}, there exists $b\in B$ such that $A':=A-a+b$ is a weighted minimum base in $M'$. Clearly, $B$ is still a weighted minimum $g$-base in $M'$ that is closest to $A'$. Moreover, $d(A',B)=|A'\setminus B|=|A|-1=r(M)-1>k$. Therefore, $G$ is not strongly $k$-close for $M'$, since $A',B$ is a witness. Because $|E'|<|E|$, we obtain a contradiction. Therefore, $r(M)\leq k+1$. Combining this with the fact that $r(M)=|A\setminus B|>k$, we get $r(M)=k+1$.
\qed
\end{proof}

As a corollary, if $G$ is not $k$-close, then there has to be a $G$-labeling $\ell$, an element $g\in G$ and a block matroid $M$ with blocks $A,B$, such that $A$ is a base, $B$ is the closest $g$-base to $A$, and $d(A,B)>k$. Similarly, we call the pair of bases $A,B$ a \emph{witness} in the unweighted setting.

\begin{corollary}\label{cor:min_counterexample_feasibility}
If a finite abelian group $G$ is not $k$-close, then the there is a block matroid $M$ of rank $k+1$ with blocks $A$ and $B$ that form a witness.
\end{corollary}
\begin{proof}
The proof follows by setting the weight function to be $0$ in the proof of \Cref{thm:mim_counterexample_weight}.
\qed
\end{proof}

Therefore, in order to show $G$ is $k$-close or strongly $k$-close, we just have to show it is $k$-close or strongly $k$-close for all block matroids of rank $k+1$, respectively.

\section{$k$-closeness and isolation}\label{sec:kclose}

We discuss progress towards conjecture \Cref{conj:feasibility} in this section. We will prove in \Cref{thm:Schrijver_Seymour_pq} that whenever a certain additive combinatorics conjecture is true, any finite abelian group $G$ is $(|G|-1)$-close.

First, we set up some notions we use later.
We say a base $B$ is \emph{isolated} under label $\ell$ if it is the unique base with the label $\ell(B)$. A labeling is called \emph{block isolating} if it isolates a base whose complement is also a base, i.e., it isolates a block. The next proposition shows that isolation and $k$-closeness are related concepts.

\begin{proposition}\label{prop:blockisolating}
If no $G$-labeling of a rank $k+1$ block matroid $M$ is block isolating, then G is $k$-close for $M$.
\end{proposition}
\begin{proof}
Suppose not. By \Cref{cor:min_counterexample_feasibility}, there exists a $G$-labeling $\ell$, $g\in G$ and witness $A,B$ that are bases of $M$ such that $B$ is the closest $g$-base to $A$. Then, $B$ is the unique base that has labeling $g$. Indeed, if there is any other base $B'\neq B$ with $\ell(B')=g$, then $d(A,B')<k+1=d(A,B)$, since the block structure guarantees that $B$ is the only farthest base from $A$, a contradiction.
\qed
\end{proof}

Equipped with \Cref{prop:blockisolating}, our goal becomes showing that block isolating labelings cannot exist.

\subsection{Congruency-constrained base with prime modulus}
 We will start by proving that \Cref{conj:feasibility} is true for any cyclic group of prime order, or equivalently congruency-constraints modulo primes. This is a special case of results for general groups which will be introduced in the next subsection. But the proof for cyclic groups of prime order is much simpler followed by a counting argument. For the sake of helping readers gain intuition, we present it as well. 
 The main tool we are going to use is the following additive combinatorics lemma by Schrijver and Seymour.
\begin{lemma}[Schrijver-Seymour \cite{schrijver1990spanning}]\label{lemma:schrijvver_Seymour}
    Let $M=(E,\mathcal{I})$ be a matroid with rank function $r$ and let $\ell:E\to \Z_p$ for some prime number $p$. Let $\ell(M):=\{\ell(B)\mid \text{$B$ is a base of $M$}\}$. Then, $|\ell(M)|\geq \min\{p,\ \sum_{g\in\Z_p} r(E(g))-r(M)+1\}$. 
\end{lemma}

The following lemma states an exchange property for matroids which is most likely routine. We also give its proof here.
\begin{lemma}\label{lemma:connectivity}
    Given a matroid $M=(E,\mathcal{I})$, let $A$ be a base, $A_1\subseteq A$, and $B_1$ be an independent set such that $A\cap B_1=\emptyset$. If $|A_1|+|B_1|-r(A_1\cup B_1)\geq t$, then there exist some $A_2\subseteq A_1$ and $B_2\subseteq B_1$ with $|A_2|=|B_2|=t$, such that $A-A_2+B_2$ is a base.
\end{lemma}
\begin{proof}
    By submodularity of matroid rank functions,
    \[
    r((A\setminus A_1)\cup B_1)+r(A_1\cup B_1)\geq r(A\cup B_1)+r(B_1)=r(M)+|B_1|.
    \]
    By assumption, $|A_1|+|B_1|-r(A_1\cup B_1)\geq t$. Combining these, we have 
    \[
    r((A\setminus A_1)\cup B_1)\geq r(M)-|A_1|+t=|A\setminus A_1|+t.
    \]
    Using the fact that $A\setminus A_1$ is independent, we deduce there exists $B_2\subseteq B_1$ with $|B_2|=t$, such that $(A\setminus A_1)\cup B_2$ is independent. Moreover, since $A$ is a base, by adding elements from $A_1$, $(A\setminus A_1)\cup B_2$ can be extended to a base which has the form $A-A_2+B_2$ for some $A_2\subseteq A_1$ with $|A_2|=|B_2|=t$. 
    \qed
\end{proof}

\begin{theorem}\label{thm:blockisolating_Zp}
    For any prime $p$, $\Z_p$ is $(p-1)$-close.
\end{theorem}
\begin{proof}
    Suppose not. Then, by \Cref{prop:blockisolating}, there exists a block matroid $M=(E,\mathcal{I})$ with $E=A\cup B$, where $A$ and $B$ are two disjoint bases with $|A|=|B|=p$, such that $A$ is isolated under some $\Z_p$-labeling $\ell$. 
    
    We claim that for any $g\in G$, $r(E(g))=|E(g)\cap A|+|E(g)\cap B|=|E(g)|$. Otherwise, letting $A_1=E(g)\cap A$, $B_1=E(g)\cap B$ and $t=1$ in \Cref{lemma:connectivity}, we can find $a\in A_1$ and $b\in B_1$ such that $A-a+b$ is a base. Since $\ell(a)=\ell(b)=g$, $\ell(A-a+b)=\ell(A)$, contradicting to the fact that $A$ is isolated.
    
    Take any element $e\in A$ and consider the matroid $M'=M\setminus e$ with ground set $E'=E\setminus e$. Note that $r(M')=r(M)$ because $B$ stays a base of $M'$. Applying \Cref{lemma:schrijvver_Seymour} to $M'$, 
    \[
    \begin{aligned}
        |\ell(M')|\geq&\min\Big\{p,\sum_{g\in\Z_p} r(E'(g))-r(M')+1\Big\}\\
        =&\min\Big\{p,\sum_{g\in\Z_p,g\neq \ell(e)} r(E(g))+r(E(\ell(e))\setminus \{e\})-r(M)+1\Big\}\\
        =&\min\Big\{p,\sum_{g\in\Z_p,g\neq \ell(e)} |E(g)|+(|E(\ell(e))|-1)-r(M)+1\Big\}\\
        =&\min\Big\{p,\sum_{g\in\Z_p} |E(g)|-p\Big\}=\min\{p,2p-p\}=p.
    \end{aligned}
    \]
    Thus, there exists a base $D$ of $M'$ such that $\ell(D)=\ell(A)$. By definition, $D$ is also a base of $M$. Since $e\in A$ but $e\notin D$, $D$ is distinct from $A$, contradicting to the fact that $A$ is isolated under label $\ell$.
    \qed
\end{proof}

\subsection{General group constraints}
To extend beyond prime modulus, we introduce the following conjecture of Schrijver and Seymour which is an extension of \Cref{lemma:schrijvver_Seymour} to any abelian group $G$.

\begin{conjecture}[Schrijver-Seymour\cite{schrijver1990spanning}, see also \cite{devos2009generalization}]\label{conj:schrijver_seymour}
    Let $M=(E,\mathcal{I})$ be a matroid with rank function $r$ and let $\ell:E\to G$ for some finite abelian group $G$. Let $\ell(M):=\{\ell(B)\mid \text{$B$ is a base of $M$}\}$. Let $H=\stab(\ell(M))$ be the stabilizer of $\ell(M)$. Then, $|\ell(M)|\geq |H|\cdot \min\big\{\sum_{Q\in G/H} r(E(Q))-r(M)+1,|G|/|H|\}$. 
\end{conjecture}
For a prime $p$ and a cyclic group $G=\Z_p$, if $\ell(M)\neq\Z_p$, then it is easy to see $\stab(\ell(M))=\{0\}$. Then, the inequality in \Cref{conj:schrijver_seymour} reduces to $\ell(M)\geq \min\big\{\sum_{g\in G}r(E(g))-r(M)+1,\ |G|\big\}$, which is precisely the form in \Cref{lemma:schrijvver_Seymour}. Otherwise, $\stab(\ell(M))=\Z_p$ and the inequality trivially holds.

\begin{theorem}\label{thm:Schrijver_Seymour_pq}
    If \Cref{conj:schrijver_seymour} is true for a finite abelian group $G$ and all of its subgroups, then $G$ is $(|G|-1)$-close.
\end{theorem}
\begin{proof}
    Let $n:=|G|$. We proceed by induction on $n$. The theorem trivially holds when $n=1$. Suppose the theorem does not hold for some $|G|=n$. Then by \Cref{prop:blockisolating}, there exists a block matroid $M=(E,\mathcal{I})$, $E=A\cup B$, where $A$ and $B$ are two disjoint bases of $M$ with $|A|=|B|=n$ such that $A$ is isolated under some labeling $\ell$. 

    Take any $e\in A$, let $M'=(E',\mathcal{I}'):=M\setminus e$. Let $H$ be the stabilizer of $\ell(M')$. Since $G$ is abelian, $H$ is a normal subgroup of $G$. Denote by $gH:=\{g+h\mid h\in H\}$ the coset of $H$ with representative $g$. First, observe that if $g\in \ell(M')$, then $gH\subseteq \ell(M')$. This is because, by definition of $H$, for any $h\in H$, $g+h\in\ell(M')$. Thus $gH\subseteq \ell(M')$. Therefore, $\ell(M')=\dot\bigcup_{g\in R}\ gH$, where $R$ is a collection of representatives of cosets of $H$. It follows that $|\ell(M')|=|R|\cdot |H|$. Let $E'(gH)=E(gH)\setminus \{e\}$ for any $g\in G$. Since $A$ is isolated under $\ell$ and $A$ is not a base of $M'$, we know $\ell(A)\notin \ell(M')$. Thus, $|\ell(M')|<|G|$. This implies $|H|< |G|$, since $|G|>\ell(M')=|R|\cdot |H|\geq |H|$ as $R\neq \emptyset$. It also follows from $|\ell(M')|<|G|$ that $|\ell(M')|\geq |H|\ (\sum_{Q\in G/H} r(E'(Q))-r(M')+1)$ in \Cref{conj:schrijver_seymour}, and thus $|R|\geq \sum_{Q\in G/H} r(E'(Q))-r(M')+1=\sum_{Q\in G/H} r(E'(Q))-n+1$. Then,
    \[
    \begin{aligned}
        &\sum_{Q\in G/H} \Big( |A\cap E'(Q)|+|B\cap E'(Q)|-r(E'(Q))\Big)\\
        ={}&(2n-1)-\sum_{Q\in G/H}r(E'(Q))\\
        \geq{}& (2n-1)-(n+|R|-1)\\
        ={}&n-|R|.
    \end{aligned}
    \]
    Therefore, \[
    \begin{aligned}
        &\max_{Q\in G/H}\Big( |A\cap E'(Q)|+|B\cap E'(Q)|-r(E'(Q))\Big)\\
        &\geq\frac{n-|R|}{|G/H|}
        >\frac{n-|G/H|}{|G/H|}
        =|H|-1,
    \end{aligned}
    \]
    where the strict inequality follows from $|\ell(M')|<|G|$.
    Suppose the maximum is attained at $Q_0=g_0H$. Then, we have 
    \[
    |A\cap E'(Q_0)|+|B\cap E'(Q_0)|-r(E'(Q_0)) \geq |H|,
    \]
    since $|A\cap E'(Q_0)|+|B\cap E'(Q_0)|-r(E'(Q_0))$ is an integer. Let $A_1:=A\cap E'(Q_0)$, $B_1:=B\cap E'(Q_0)$. By \Cref{lemma:connectivity}, there exists a base $A-A_2+B_2$ such that $A_2\subseteq A_1$,  $B_2\subseteq B_1$ and $|A_2|=|B_2|=|H|$.

    Let $M''=(E'',\mathcal{I}'')$ be the matroid obtained by contracting $A\setminus A_2$ and deleting $B\setminus B_2$, i.e. $M''=M/\big(A\setminus A_2\big)\setminus \big(B\setminus B_2\big)$.
    Note that $M''$ is a rank $|H|$ block matroid with blocks $A_2$ and $B_2$.
    Consider an $H$-labeling on $E''=A_2\cup B_2$, $\ell'':E''\rightarrow H$, such that $\ell''(e)=\ell(e)-g_0\in H, \forall e\in E''$. Note that $|H|<|G|$. By the induction hypothesis and the assumption that \Cref{conj:schrijver_seymour} is true for the subgroup $H$ of $G$, we know that there is no isolating $H$-labeling of $M''$. Thus, there must be another base $D_2\neq A_2$ of $M''$, such that $\ell''(D_2)=\ell''(A_2)$. Thus $\ell(D_2)=\ell''(D_2)+|H|\cdot g_0 =\ell''(A_2)+|H|\cdot g_0=\ell(A_2)$. By the definition of contraction, $D:=D_2\cup (A\setminus A_2)$ is a base of matroid $M$ which is distinct from $A$. And $\ell(D)=\ell(D_2)+\ell(A\setminus A_2)=\ell(A_2)+\ell(A\setminus A_2)=\ell(A)$, contradicting to the fact that $A$ is isolated under label $\ell$.
    \qed
\end{proof}

The following lemma is the current status showing on which group the Schrijver-Seymour conjecture is true.
\begin{lemma}[\cite{devos2009generalization}]
    \Cref{conj:schrijver_seymour} is true for $G$ if $|G|=pq$ for primes $p,q$ or $G=\Z_{p^n}$ for a prime $p$ and a positive integer $n$.
\end{lemma}

\begin{corollary}\label{cor:closeness}
    $G$ is $(|G|-1)$-close if $|G|=pq$, or $G=\Z_{p^n}$ for any $p,q$ primes and $n$ a positive integer.
\end{corollary}

\section{Strong $k$-closeness}\label{sec:stronglykclose}
This section discusses our results for strong $k$-closeness. Unfortunately, we have little understanding of strong $k$-closeness. Hence, we focus on restricted class of matroids and small groups. 

For the matroids and groups concerned within this section, we obtain a tighter bounds than \Cref{conj:optimization} implies. Instead of $(|G|-1)$-closeness in the conjecture, we have $(D(G)-1)$-closeness.

\subsection{Strongly base orderable matroids}
We first study strongly base orderable matroids. 
\begin{theorem}\label{thm:stronglyorderable}
Every $G$-labeling is strongly $(D(G)-1)$-close for strongly base orderable matroids.
\end{theorem}
\begin{proof}
Suppose not. By noting that strongly base orderability is closed under taking minors \cite{ingleton1977transversal}, it is not hard to see from \Cref{thm:mim_counterexample_weight} that there exists a counterexample which is a rank $D(G)$ strongly base orderable block matroid $M$ with blocks $A$ and $B$ as a witness. Let $\ell$ be a $G$-labeling of $M$, $w$ be a weight function such that $A$ is an optimum base and $B$ is the optimum $g$-base closest from $A$.

Let $f:B\to A$ be the bijection obtained from strongly base orderability.
Define $\ell'(X) = \ell(f(X))-\ell(X)$ for all $X\subseteq B$. Since $|B|=D(G)$, by the definition of Davenport's constant, there is $B_1\subseteq B$ such that $\ell'(B_1) = 0$. Let $A_1=f(B_1)$. This implies $\ell(A_1)=\ell(B_1)$ and thus the base $B-B_1+A_1$ has the same label as $B$, i.e. $B-B_1+A_1$ is also a $g$-base.

Now, consider $A_2=A-A_1$ and $B_2=B-B_1$. One has $f(B_2)=A_2$ and therefore $B-B_2+A_2=A-A_1+B_1$ is also a base. Since $A$ is an optimum base, $w(A-A_1+B_1)\geq w(A)$, which means $w(A_1)\leq w(B_1)$. Therefore, $w(B-B_1+A_1) \leq w(B)$. This means $B-B_1+A_1$ is also an optimum $g$-base. But $A_1\subsetneqq A$ since $\ell(A)\neq \ell(B)$, which means $B-B_1+A_1$ is an optimum $g$-base closer to $A$, a contradiction.
\qed
\end{proof}

In fact, we can define a weaker property. Two bases $A,B$ are $k$-replaceable if there exists a bijection $f:B\setminus A\to A\setminus B$ such that for any subset $B'\subseteq B\setminus A$ of size at most $k$, $B-B'+f(B')$ is a base. We say a matroid is $k$-replaceable if every pair of bases are $k$-replaceable. We have that \Cref{thm:stronglyorderable} also holds for $(D(G)-1)$-replaceable matroids.

\subsection{Small groups}
Next, we consider small groups. \Cref{thm:mim_counterexample_weight} shows one only needs to check block matroids of rank $D(G)$ to know if $G$ is strongly $(D(G)-1)$-close. 

For $G=\Z_2$ which has $D(G)=2$, observe that all size $4$ matroids are strongly base orderable, hence $\Z_2$ is $1$-close. For $G\in \{\Z_3, \Z_2\times \Z_2\}$ which has $D(G)=3$, we need to check all rank $3$ block matroids. Since all but $M(K_4)$, the graphic matroid on $K_4$, are strongly base orderable \cite{MAYHEW2008415}, we only need to test a single matroid. However, testing strongly $2$-closeness of a single matroid is still computational intensive. 

Instead, we introduct the notion of strong block isolation. A $G$-labeling $\ell$ \emph{strongly isolates} a block $B$, if it is the unique block with label $\ell(B)$. A $G$-labeling is strong block isolating, if it strongly isolates a block.

\begin{proposition}\label{prop:strongblockisolation}
    If no $G$-labeling of a rank $k+1$ block matroid $M$ is strong block isolating, then G is strongly $k$-close for $M$.
\end{proposition}
\begin{proof}
Suppose not. By \Cref{thm:mim_counterexample_weight}, there exists a $G$-labeling $\ell$, $g\in G$ and witness $A, B$ that are bases of $M$ such that $B$ is the closest optimum $g$-base to $A$. Fix such labeling $\ell$. Since $\ell$ is not strong block isolating, there is an non-empty $A_1\subseteq A$ and $B_1\subseteq B$ such that $A-A_1+B_1$ is a block with the same label as $A$. Let $A_2 = A\setminus A_1$, $B_2 = B\setminus B_1$. Then the complement of $A-A_1+B_1$ is the base $A-A_2+B_2$, and has the same label as $B$. Since $w(B)=w(A)-w(A_1)+w(B_1)-w(A_2)+w(B_2)\geq w(A)-w(A_2)+w(B_2) = w(A-A_2+B_2)$, $A-A_2+B_2$ is a closer optimum $g$-base to $A$, a contradiction.
\qed
\end{proof}

Hence, for each $G$ in hand, we show no $G$-labeling of a rank $D(G)$ block matroid is strong block isolating, which implies $G$ is strongly $(D(G)-1)$-close. We tested $M(K_4)$ for $\Z_3$ and $\Z_2\times \Z_2$ and it turned out that they are both strongly $2$-close. Further, we tested for each $\Z_4$-labeling of a rank $4$ block matroid. There are $940$ non-isomorphic rank $4$ matroids of size $8$ \cite{matsumoto2012}, and for each one that is also a block matroid, we test $4^8 = 65536$ different labelings. None of them is strongly block isolating. Therefore, we get the following.

\begin{proposition}
    $G$ is strongly $(D(G)-1)$-close if $|G|\leq 4$.
\end{proposition}

\subsubsection*{Acknowledgements.} We thank the reviewers for their valuable feedback and Andr\'{a}s Imolay for noticing and correcting an error in \Cref{thm:Schrijver_Seymour_pq}.

\bibliographystyle{plain}
\bibliography{reference}
\end{document}